\documentclass[12pt]{amsart}
\usepackage{amssymb,hyperref}
\usepackage[all]{xy}

\theoremstyle{plain}
\newtheorem{theorem}{Theorem}[section]
\newtheorem{proposition}[theorem]{Proposition}
\newtheorem{corollary}[theorem]{Corollary}
\newtheorem{lemma}[theorem]{Lemma}

\theoremstyle{definition}

\newtheorem{example}[theorem]{Example}

\newtheorem{question}[theorem]{Question}

\newcommand{\abs}[1]{\lvert#1\rvert}
\newcommand{\norm}[1]{\lVert#1\rVert}
\newcommand{\bigabs}[1]{\bigl\lvert#1\bigr\rvert}

\newcommand{\term}[1]{{\textit{\textbf{#1}}}}

\renewcommand{\mid}{\::\:}

\DeclareMathOperator{\Int}{Int}

\DeclareSymbolFont{bbold}{U}{bbold}{m}{n}
\DeclareSymbolFontAlphabet{\mathbbold}{bbold}
\def\one{\mathbbold{1}}

\begin{document}

\title[Locally piecewise affine functions]
{Locally piecewise affine functions and their order structure}

\author{S. Adeeb}
\address{Department of Civil and Environmental Engineering,
University of Alberta, Edmonton, AB, T6G\,1H9, Canada.}
\email{adeeb@ualberta.ca}

\author{V.~G. Troitsky}
\address{Department of Mathematical and Statistical Sciences,
         University of Alberta, Edmonton, AB, T6G\,2G1, Canada.}
\email{troitsky@ualberta.ca}

\thanks{The authors were supported by NSERC grants}
\date{\today}

\begin{abstract}
  Piecewise affine functions on subsets of $\mathbb R^m$ were studied
  in
  \cite{Ovchinnikov:02,Aliprantis:06a,Aliprantis:07a,Aliprantis:07}. In
  this paper we study a more general concept of a locally piecewise
  affine function. We characterize locally piecewise affine functions
  in terms of components and regions. We prove that a positive
  function is locally piecewise affine iff it is the supremum of a
  locally finite sequence of piecewise affine functions. We prove
  that locally piecewise affine functions are uniformly dense in
  $C(\mathbb R^m)$, while piecewise affine functions are sequentially
  order dense in $C(\mathbb R^m)$.

  This paper is partially based on \cite{Adeeb:14}.
\end{abstract}

\maketitle

\section{Piecewise affine functions}

We start with a brief overview of affine and piecewise affine
functions; we refer the reader to \cite{Ovchinnikov:02,Aliprantis:06a}
and \cite[Chapter 7]{Aliprantis:07} for details.

Fix $m\in\mathbb N$. By an \term{affine} function we mean a function
$f\colon\mathbb R^m\to\mathbb R$ of form $f(x)=\langle v, x\rangle+b$ for some
$v\in\mathbb R^m$ and $b\in\mathbb R$. Its kernel $\{f=0\}$ is an
\term{affine hyperplane} in $\mathbb R^m$.
It is easy to see
that any two affine functions which agree on a non-empty open set are
equal. We write $\mathcal A$ for the vector space of all affine
functions on $\mathbb R^m$. Clearly, $\mathcal A$ is a subspace of
$C(\mathbb R^m)$, the space of all real-valued continuous functions on
$\mathbb R^m$.

Let $\Omega$ be a convex closed subset of $\mathbb R^m$ with non-empty
interior. Following \cite{Aliprantis:07}, we call such a set a
\term{solid domain}. A continuous function
$f\colon\Omega\to\mathbb R$ is called \term{piecewise affine} if
it agrees with finitely many affine functions. That is, there exist
$f_1\dots,f_n\in\mathcal A$ such that for every $x\in\Omega$
there exists $i\le n$ with $f(x)=f_i(x)$. Equivalently (see
Definition~7.10 and Theorem~7.13 in~\cite{Aliprantis:07}), a function
$f\colon\Omega\to\mathbb R$ is piecewise affine iff there exist
distinct affine functions $f_1,\dots,f_n$ and subsets $S_1,\dots,S_n$
of $\Omega$ such that
\begin{enumerate}
  \item on $S_i$, $f$ agrees with $f_i$;  
  \item\label{pa-perfect} $S_i=\overline{\Int S_i}\ne\varnothing$ for every $i$;
  \item $\Int S_i\cap\Int S_j=\varnothing$ whenever $i\ne j$;
  \item $\bigcup_{i=1}^nS_i=\Omega$.
\end{enumerate}
In particular, it follows from~\eqref{pa-perfect} that the set
$\{f=f_i\}$ has non-empty interior for every $i=1,\dots,n$. The sets
$\{f_i\}_{i=1}^n$, $\{S_i\}_{i=1}^n$, and
$\bigl\{(S_i,f_i)\bigr\}_{i=1}^n$ are called the \term{components},
\term{regions}, and \term{characteristic pairs} of $f$,
respectively. The characteristic pairs of $f$ are unique up to
re-ordering. We write $\mathcal{PA}(\Omega)$ for the space of all
piecewise affine functions on $\Omega$. We put
$\mathcal{PA}=\mathcal{PA}(\mathbb R^m)$.

Suppose $f\in\mathcal PA(\Omega)$ with characteristic pairs
$\bigl\{(S_i,f_i)\bigr\}_{i=1}^n$. For each pair $(i,j)$ with $1\le
i,j\le n$ and $i\ne j$, the set
\begin{math}
  \bigl\{x\in\mathbb R^m\mid f_i(x)=f_j(x)\bigr\},
\end{math}
or $\{f_i=f_j\}$ is a hyperplane in $\mathbb R^m$. Let $C$ be the
union of all such hyperplanes. Then $C$ is a closed set with empty
interior. It is easy to see that the set $(\Int\Omega)\setminus C$ is
open, dense in $\Omega$, and naturally
partitioned into a finite number of convex open subsets. We call these
subsets the \term{cells} of $f$.

Let $K$ be a cell of $f$. Since $K$ is disjoint from $C$, no two
components agree at any point of $K$. It follows that there exists a
component $f_i$ such that $f$ agrees with $f_i$ on $K$, so that
$K\subseteq S_i$. It also follows that for every $j\ne i$ we either
have $f_j(x)<f(x)$ for all $x\in K$ or $f_j(x)>f(x)$ for all $x\in K$.

\begin{example}
  Let $m=1$ and $f(t)=\min\bigl\{\abs{t},1\bigr\}$. Then $f$ is
  piecewise affine on $\mathbb R$, with components $f_1(t)=t$,
  $f_2(t)=-t$, and $f_3(t)=1$, and with regions $S_1=[0,1]$,
  $S_2=[-1,1]$, and $S_3=(-\infty,-1]\cup[1,+\infty)$. Note that $S_3$
  is not connected. The cells of $f$ are $K_1=(-\infty,-1)$,
  $K_2=(-1,0)$, $K_3=(0,1)$, and $K_4=(1,+\infty)$.
\end{example}

Recall that under pointwise order, the space $C(\Omega)$ is a vector
lattice. It is easy to see that $\mathcal{PA}(\Omega)$ is a (vector)
sublattice of $C(\Omega)$.

Let $A$ be a subset of a vector lattice $V$. Recall that
  \index{$A^\vee$}\index{$A^\wedge$}
  \begin{eqnarray*}
    A^\vee&=&\bigl\{x_1\vee\dots\vee x_n\mid
      n\in\mathbb N,\ x_1,\dots,x_n\in A\bigr\}
    \quad\mbox{and}\\
    A^\wedge&=&\bigl\{x_1\wedge\dots\wedge x_n\mid
      n\in\mathbb N,\ x_1,\dots,x_n\in A\bigr\}.
  \end{eqnarray*}
If $A$ is a linear subspace of $V$ then
$A^{\vee\wedge}=A^{\wedge\vee}$ and this set is exactly the sublattice
generated by $A$ in $V$; see, e.g., \cite[Lemma~1.21]{Aliprantis:07}.

Since $\mathcal A\subseteq\mathcal{PA}$ and $\mathcal{PA}$ is a
sublattice of $C(\mathbb R^m)$, it follows that
$\mathcal A^{\wedge\vee}\subseteq\mathcal{PA}$. In fact,
\cite[Theorem~7.23]{Aliprantis:07} asserts that
$\mathcal{PA}=\mathcal A^{\wedge\vee}$, hence $\mathcal{PA}$ is
exactly the sublattice of $C(\mathbb R^m)$ generated by $\mathcal
A$. Moreover, if $\Omega$ is a solid domain in $\mathbb R^m$ and
$f\in\mathcal{PA}(\Omega)$ with component set $F=\{f_i\}_{i=1}^n$ then
$f\in F^{\wedge\vee}$ or, more precisely, $f$ agrees with the
restriction to $\Omega$ of a function in $F^{\wedge\vee}$. Since every
function in $F^{\wedge\vee}$ is defined on all of $\mathbb R^m$, it
follows that every function in $\mathcal{PA}(\Omega)$ extends to a
function in $\mathcal{PA}$.

We write $B_\infty^m$ for the closed unit ball of
$\ell_\infty^m$:
\begin{displaymath}
  B_\infty^m=\bigl\{x\in\mathbb R^m\mid\max\limits_{1\le i\le
    m}\abs{x_i}\le 1\bigr\}. 
\end{displaymath}
For $n\in\mathbb N$, we write $\Omega_n$ for the ball of radius $n$
centred at zero in $\ell_\infty^m$; i.e., $\Omega_n=nB_\infty^m$.

\section{Locally piecewise affine functions}

A function $f\colon\mathbb R^m\to\mathbb R$ is said to be
\term{locally piecewise affine} if its restriction to any bounded
solid domain $\Omega$ is in $\mathcal{PA}(\Omega)$. The following
lemma is straightforward.

\begin{lemma}\label{on-balls}
  Let $f\colon\mathbb R^m\to\mathbb R$. Then $f$ is locally piecewise
  affine iff the restriction of $f$ to $\Omega_n$ is in
  $\mathcal{PA}(\Omega_n)$ for every $n\in\mathbb N$.
\end{lemma}

We write $\mathcal{LPA}$ for the collection of all locally piecewise
affine functions on $\mathbb R^m$. Clearly, $\mathcal{LPA}$ is a
sublattice of $C(\mathbb R^m)$ and
$\mathcal A\subsetneq\mathcal{PA}\subsetneq\mathcal{LPA}$.

\bigskip

Let $F\subset C(\mathbb{R}^m)$. We say that $F$ is \term{locally
  finite} if for every compact set $C$ all but finitely many functions
in $F$ vanish on $C$. That is, the set
\begin{math}
  \bigl\{f\in F\mid\exists x\in C\ f(x)\ne 0\bigr\}
\end{math}
is finite. Just as in Lemma~\ref{on-balls}, it is easy to see that $F$
is locally finite iff for every $n$ all but finitely many functions in
$F$ vanish on $\Omega_n$. Clearly, any subset of a locally finite
set of functions is again locally finite. Additionally, if
$F\subset C(\mathbb{R}^m)$ is a locally finite set of functions, then
$F^\vee$, $F^\wedge$ and consequently $F^{\vee\wedge}$ are all locally
finite. If $F$ is a sequence, then $F$ is termed a \term{locally finite
sequence of functions}.

It is a standard fact that $C(\mathbb R^m)$ is not order complete. Yet
we have the following.

\begin{lemma}\label{sup-lf}
  Every locally finite subset of $\mathcal{PA}$ has a supremum and an
  infimum in $C(\mathbb R^m)$, which belong to $\mathcal{LPA}$ and
  agree with the pointwise maximum and minimum, respectively.
\end{lemma}

\begin{proof}
  Let $F\subseteq\mathcal{PA}$ be a locally finite set. For every
  point $x\in\mathbb R^m$, all but finitely many members of $F$ vanish
  at $x$. Hence, $h(x)=\max\bigl\{f(x)\mid f\in F\bigr\}$ is
  defined. Similarly, on every open ball $B$ in $\mathbb R^m$, $h$ is
  the pointwise maximum of a finite subcollection of $F$, so that $h$ is
  continuous on $B$; it follows that $h\in C(\mathbb R^m)$. It follows
  from the definition that $h=\sup F$ in $C(\mathbb R^m)$. It also
  follows that on every ball, $h$ agrees with a finite subset of $F$,
  hence $h$ is locally piecewise affine.  The proof for the infimum
  is similar.
\end{proof}

We write $\one$ for the constant one function on $\mathbb R^m$.

\begin{lemma}\label{affine-Uryson}
  For every $\varepsilon>0$ there exists
  $f\in\mathcal{PA}$ such that $0\le f\le\one$, $f$
  equals $1$ on $B_\infty^m$, and vanishes outside of
  $(1+\varepsilon)B_\infty^m$.
\end{lemma}

\begin{proof}
  Put $g(x)=\norm{x}_{\infty}=\max_{i\le m}\abs{x_i}$ for
  $x\in\mathbb R^m$. Put $h\colon\mathbb R_+\to\mathbb R_+$ so that
  $h$ is continuous, equals $1$ on $[0,1]$, vanishes on
  $[1+\varepsilon,+\infty)$ and is affine on $[1,1+\varepsilon]$. Let
  $f=h\circ g$. Being the composition of two piecewise affine
  functions, $f$ is itself piecewise affine. It is easy to verify that
  $f$ satisfies the rest of the requirements.
\end{proof}

\begin{theorem}\label{lf-sup}
  For every $0\le f\in\mathcal{LPA}$ there exists a locally finite
  sequence $(f_i)$ in $\mathcal{PA}_+$ such that
  $f=\sup f_i$.
\end{theorem}

\begin{proof}
  Let $(c_n)$ be an enumeration of all integer points in $\mathbb
  R^m$. Consider the ``boxes'' $K_n=c_n+B_\infty^m$ and $\widetilde
  K_n=c_n+2B_\infty^m$.

  Fix $n$. Since $f$ is continuous, it is bounded on $K_n$. Let $M>0$
  be such that $f(x)<M$ for all $x\in K_n$. As in
  Lemma~\ref{affine-Uryson}, we can find a piecewise affine function
  $g_n$ such that $0\le g_n\le M\one$, $g_n$ equals $M$ on $K_n$, and $g_n$
  vanishes outside $\widetilde K_n$. Put $f_n=f\wedge g_n$. We
  will show that the sequence $(f_n)$ satisfies all the requirements.

  By construction, $0\le f_n\le f$, $f_n$ agrees with $f$ on $K_n$ and
  vanishes outside $\widetilde K_n$.  Observe that $f_n$ is piecewise
  affine because on $\widetilde K_n$ it agrees with $f\wedge g_n$, which
  is piecewise affine on $\widetilde K_n$ and $f_n$ vanishes outside
  of $\widetilde K_n$. The sequence $(f_n)$ is locally finite because
  for every $n$ the only members of the sequence that do not vanish on
  $K_n$ are $f_n$ itself and the members corresponding to the
  (finitely many) boxes that meet $\widetilde K_n$. Finally,
  $f=\sup f_n$ because $0\le f_n\le f$ and $f_n$ agrees with $f$ on
  $K_n$ for every $n$.
\end{proof}

\section{Characteristic pairs of locally piecewise functions}

Recall that every piecewise affine function is determined by its components
and regions. The following theorem provides a similar characterization
for locally piecewise affine functions. 

\begin{theorem}
  A function $f\colon\mathbb R^m\to\mathbb R$ is locally piecewise
  affine iff there exist a sequence of distinct affine functions
  $(f_i)$ and a sequence $(S_i)$ of subsets of $\mathbb R^m$ such that
 \begin{enumerate}
  \item\label{lpa-agree} on $S_i$, $f$ agrees with $f_i$;  
  \item\label{lpa-perfect} $S_i=\overline{\Int S_i}\ne\varnothing$ for every $i$;
  \item\label{lpa-disjoint} $\Int S_i\cap\Int S_j=\varnothing$ whenever $i\ne j$;
  \item\label{lpa-union} $\bigcup_{i=1}^\infty S_i=\mathbb R^m$;
  \item\label{lpa-finite} every bounded set meets only finitely many $S_i$'s.
 \end{enumerate}
\end{theorem}

\begin{proof}
  Suppose that $f$ satisfies the five conditions; show that it is
  locally piecewise affine. First, observe that $f$ is
  continuous. Indeed, suppose that $x_n\to x$ in $\mathbb R^m$ but
  $f(x_n)\not\to f(x)$. Passing to a subsequence, we can find an
  $\varepsilon>0$ such that $\bigabs{f(x_n)-f(x)}>\varepsilon$ and
  $x_n\in B(x,1)$ for every $n$. By~\eqref{lpa-finite}, the ball
  $B(x,1)$ meets only finitely many $S_i$'s. Hence, passing to a
  further subsequence, we may assume that the sequence $(x_n)$ is
  contained in $S_i$ for some $i$. Since $S_i$ is closed, we have
  $x\in S_i$. It follows that $f(x_n)=f_i(x_n)\to f_i(x)=f(x)$; a
  contradiction. Thus, $f$ is continuous. By~\eqref{lpa-finite} again,
  every bounded solid domain $\Omega$ is contained in the union of finitely
  many $S_i$'s; it follows that the restriction of $f$ to $\Omega$ is
  in $\mathcal{PA}(\Omega)$, so that $f$ is locally piecewise
  affine.

  Suppose now that $f$ is locally piecewise affine. Fix
  $n\in\mathbb N$. The restriction of $f$ to $\Omega_n$ is in
  $\mathcal{PA}(\Omega_n)$; consider the affine components of this
  restriction. By collecting the affine functions associated with
  $\Omega_n$ for all $n\in\mathbb N$, we produce a sequence $(f_i)$
  in $\mathcal A$ such that for every $n$ there exists $k_n$ such that
  $f$ agrees with $f_1,\dots,f_{k_n}$ on $\Omega_n$. By dropping
  repeated terms, we may assume without loss of generality that all
  the members in this sequence are distinct.

  For every $i$, put $U_i=\Int\{f=f_i\}$ and
  $S_i=\overline{U_i}$. Since $f_i$ is a component of the
  restriction of $f$ to $\Omega_n$ for some $n$, it follows that
  $U_i\ne\varnothing$. It is easy to see that \eqref{lpa-agree}
  and~\eqref{lpa-perfect} are satisfied. To show~\eqref{lpa-disjoint},
  suppose that $\Int S_i$ meets $\Int S_j$ for some $i$ and $j$. This
  follows that $f_i$ and $f_j$ agree on an open set, hence $f_i=f_j$
  and, therefore, $i=j$.

  Fix $x\in\mathbb R^m$ and take $n\in\mathbb N$ such that
  $x\in\Int(\Omega_n)$. Let $K_1,\dots,K_l$ be the cells
  generated in $\Omega_n$ by $f_1,\dots,f_{k_n}$. Since
  $\bigcup_{j=1}^lK_p$ is dense in $\Omega_n$, we have
  $x\in\overline{K_p}$ for some $p\le l$. Note that $K_p$ is open and
  $f$ agrees with some $f_i$ on $K_p$. It follows that
  $K_p\subseteq U_i$ and, therefore, $x\in S_i$. This
  yields~\eqref{lpa-union}.

  Finally, we will show that $\Int(\Omega_n)$ only meets
  $S_1,\dots,S_{k_n}$; this will immediately yield~\eqref{lpa-finite}.
  Suppose that $\Int(\Omega_n)$ meets $S_i$ for some
  $i\in\mathbb N$. It follows that $\Int(\Omega_n)$ meets
  $U_i$. This implies that $U_i$ meets $K_r$ for some $r\le l$. Note
  that $f$ agrees with $f_j$ on $K_r$ for some $j\le k_n$. It follows
  that $f_j$ and $f_i$ agree on the non-empty open set $K_r\cap U_i$,
  hence $j=i$ and, therefore, $i\le k_n$.
\end{proof}

As before, the sets $\{f_i\}_{i=1}^\infty$, $\{S_i\}_{i=1}^\infty$, and
$\bigl\{(S_i,f_i)\bigr\}$ are called the \term{components},
\term{regions}, and \term{characteristic pairs} of $f$,
respectively.

\section{$\mathcal{LPA}$ is uniformly dense in $C(\mathbb R^m)$}

It is of interest whether every continuous function on $\mathbb R^m$
can be approximated in some sense by piecewise affine functions. We
will consider two kinds of denseness here: uniform denseness and order
denseness. We start with uniform denseness. It is easy to see that
$\mathcal{PA}$ is not uniformly dense in $C(\mathbb R^m)$ even for
$m=1$.

\begin{theorem}
  $\mathcal{LPA}$ is uniformly dense in $C(\mathbb R^m)$. That is, for
  every $f\in C(\mathbb R^m)$ and every $\varepsilon>0$ there exists
  $h\in\mathcal{LPA}$ such that $\abs{f-h}\le\varepsilon\one$.
\end{theorem}

\begin{proof}
  The proof is analogous to that of Theorem~\ref{lf-sup}. Let
  $f\in C(\mathbb R^m)$ and $\varepsilon>0$. 

  First, we consider the case when $f\ge 0$.  Let $(c_n)$ be an
  enumeration of all integer points in $\mathbb R^m$. Consider the
  ``boxes'' $K_n=c_n+B_\infty^m$ and $\widetilde K_n=c_n+2B_\infty^m$.

  Fix $n$. Observe that $\mathcal{PA}(\widetilde K_n)$ is a sublattice of
  $C(\widetilde K_n)$. It clearly contains constant functions. It is
  easy to see that it separates points. Indeed, let
  $x,y\in\widetilde K_n$ with $x\ne y$. Then $x_j\ne y_j$ for some
  $j\le m$. That is, $e_j^*(x)\ne e_j^*(y)$,
  where $e_j^*$ is the $j$-th coordinate functional. Clearly, $e_j^*$ is an
  affine function. Since $\widetilde K_n$ is compact, it now follows
  from the lattice version of Stone-Weierstrass Theorem that
  $\mathcal{PA}(\widetilde K_n)$ is uniformly dense in
  $C(\widetilde K_n)$, see, e.g., Theorem~11.3 on page~88
  of~\cite{Aliprantis:98}. It follows that there is
  $h_n\in\mathcal{PA}(\widetilde K_n)$ such that
  $\bigabs{h_n(x)-f(x)}<\varepsilon$ for all $x\in\widetilde
  K_n$. Extend $h_n$ to a piecewise affine function $\tilde h_n$
  in $\mathcal{PA}$. 

  Since $h_n$ is continuous on $\widetilde K_n$, it is bounded on
  $\widetilde K_n$. Let $M_n>0$ be such that $h_n(x)<M_n$ for all
  $x\in\widetilde K_n$. As in Lemma~\ref{affine-Uryson}, we can find
  $g_n\in\mathcal{PA}$ such that $0\le g_n\le M_n\one$, $g_n$ equals
  $M_n$ on $K_n$, and $g_n$ vanishes outside $\widetilde K_n$.

  Consider the sequence of functions
  $\bigl(\tilde h_n\wedge g_n\bigr)$. Since $g_n$ vanishes outside of
  $\widetilde K_n$, all but finitely many terms of this sequence
  vanish on $K_n$; it follows that the sequence
  $\bigl(\tilde h_n\wedge g_n\bigr)$ is locally finite. Let
  $h=\sup_n\tilde h_n\wedge g_n$; by Lemma~\ref{sup-lf} this supremum
  exists, agrees with the pointwise maximum, and belongs to $\mathcal{LPA}$.

  For every $n$ and every $x\in K_n$, we have 
  \begin{displaymath}
    h(x)\ge(\tilde h_n\wedge g_n)(x)=h_n(x)\wedge M_n=h_n(x)\ge f(x)-\varepsilon.
  \end{displaymath}
  It follows that $h\ge f-\varepsilon\one$, so that $h-f\ge
  -\varepsilon\one$.

  On the other hand, for every $n$ and every $x\in\widetilde K_n$ we
  have
  \begin{displaymath}
     (\tilde h_n\wedge g_n)(x)\le h_n(x)\le f(x)+\varepsilon. 
  \end{displaymath}
  Since $\tilde h_n\wedge g_n\le g_n$, it vanishes outside of
  $\widetilde K_n$. It follows that
  $\tilde h_n\wedge g_n\le f+\varepsilon\one$. Taking sup over $n$ we
  get $h\le f+\varepsilon\one$, so that
  $h-f\le\varepsilon\one$. Therefore, $\abs{h-f}\le\varepsilon\one$.

  Suppose now that $f$ is an arbitrary function in $C(\mathbb
  R^m)$. Then $f=f^+-f^-$, with $0\le f^+,f^-\in C(\mathbb R^m)$. By
  the first part of the proof, we find $h_1$ and $h_2$ in
  $\mathcal{LPA}$ such that
  $\abs{f^+-h_1}\le\frac{\varepsilon}{2}\one$ and
  $\abs{f^--h_2}\le\frac{\varepsilon}{2}\one$. Put $h=h_1-h_2$, then
  \begin{displaymath}
    \abs{f-h}\le\abs{f^+-h_1}+\abs{f^--h_2}\le\varepsilon\one.
  \end{displaymath}
\end{proof}

\section{$\mathcal{PA}$ is $\sigma$-order dense in $C(\mathbb R^m)$}

Recall that a sublattice $F$ of an Archimedean vector lattice $E$ is
said to be \term{order dense} if for every $0<x\in E$ there exists
$y\in F$ with $0<y\le x$. Equivalently, for every $0<x\in E$ we have
$x=\sup A$ where $A=\{y\in F\mid 0\le y\le x\}$; see, e.g.,
\cite[Theorem~1.34]{Aliprantis:06}. Note that the set $A$ may be
viewed as an increasing net. Therefore, $F$ is order dense in $E$ iff
every element of $E_+$ is the supremum of an increasing net in $F_+$.

Let $f\in C(\mathbb R^m)$. Suppose that $f>0$; that is, $f(x)\ge 0$
for every $x\in\mathbb R^m$ and $f(c)>0$ for some $c\in\mathbb R^m$.
It follows easily from Lemma~\ref{affine-Uryson} that there exists a
locally piecewise function $h$ such that $0<h\le f$. This means that
$\mathcal{PA}$ is an order dense sublattice in $C(\mathbb R^m)$.  It
follows that if $0\le f\in C(\mathbb R^m)$ then $f$ is the supremum of
an increasing net $(f_\alpha)$ in $\mathcal{PA}$, that is,
$f_\alpha\uparrow f$. This leads to two natural questions.

\begin{question}\label{q:seq}
  Can we replace a net with a sequence?
\end{question}

\begin{question}\label{q:non-pos}
  How do we deal with non-positive functions?
\end{question}

The answer to the first question is related to the countable sup
property. Recall that a vector lattice $E$ has the \term{countable sup
  property} if every increasing net $(x_\alpha)$ in $E$ with
$x=\sup x_\alpha$ has an increasing subsequence with the same
supremum. Thus, to answer the first question, it suffices to prove
that $C(\mathbb R^m)$ has the countable sup property.

Recall that for a sequence $(x_\alpha)$ in a vector lattice $E$, the
notation $x_\alpha\uparrow$ means $(x_\alpha)$ is increasing;
$x_\alpha\uparrow\le x$ means $x_\alpha\uparrow$ and $x_\alpha\le x$
for every $\alpha$, and $x_\alpha\uparrow x$ means $x_\alpha\uparrow$
and $x=\sup_\alpha x_\alpha$. Note that $f_\alpha\uparrow f$ in
$C(\Omega)$ does not imply pointwise convergence. For
$f,g\in C(\mathbb R^m)$, the notation $f<g$ means $f\le g$ and
$f\ne g$. For $f\in C(\mathbb R^m)$ and a solid domain $\Omega$ in
$\mathbb R^m$, we write $f_{|\Omega}$ for the restriction of $f$ to
$\Omega$; we view it as an element of $C(\Omega)$. 

\begin{lemma}\label{restr}
  Let $(f_\alpha)$ be a net in $C(\mathbb R^m)$ and
  $f\in C(\mathbb R^m)$.  Then $f_\alpha\uparrow f$ in
  $C(\mathbb R^m)$ iff ${f_\alpha}_{|\Omega_n}\uparrow f_{|\Omega_n}$
  in $C(\Omega_n)$ for every $n$.
\end{lemma}

\begin{proof}
  Suppose that $f_\alpha\uparrow f$ in $C(\mathbb R^m)$. Fix $n$.  It
  is clear that $({f_\alpha}_{|\Omega_n})$ is an increasing net in
  $C(\Omega_n)$ and $f_{|\Omega_n}$ is an upper bound of this
  net. Suppose it is not the least upper bound. Then there exists
  $g\in C(\Omega_n)$ such that ${f_\alpha}_{|\Omega_n}\le g$ for every
  $\alpha$ but $g(t_0)<f(t_0)$ for some $t_0\in\Omega_n$. Since
  $\Omega_n$ is a solid domain, we may assume that
  $t_0\in\Int\Omega_n$. It follows easily that there exists a
  neighborhood $V$ of $t_0$ and $\varepsilon>0$ such that
  $V\subseteq\Int\Omega_n$ and $g(t)<f(t)-\varepsilon$ for all
  $t\in V$. Find $0<h\in C(\mathbb R^m)$ such that $h(t_0)>0$,
  $h(t)<\varepsilon$ for all $t\in V$, and $h$ vanishes outside
  $V$. Then $f-h<f$, yet $f_\alpha\le f-h$ for every $\alpha$; this
  contradicts $f=\sup f_\alpha$.

  Conversely, suppose that ${f_\alpha}_{|\Omega_n}\uparrow f_{|\Omega_n}$
  in $C(\Omega_n)$ for every $n$. First, we claim that
  $f_\alpha\uparrow\le f$ in $C(\mathbb R^m)$. Indeed, fix any
  $t_0\in\mathbb R^m$. Find $n$ so that $t_0\in\Omega_n$. It follows
  from ${f_\alpha}_{|\Omega_n}\uparrow f_{|\Omega_n}$ that
  $f_\alpha(t_0)\uparrow\le f(t_0)$. In particular, $f$ is an upper
  bound of $(f_\alpha)$ in $C(\mathbb R^m)$.

  On the other hand, suppose that $g$ is an upper bound of
  $(f_\alpha)$ in $C(\mathbb R^m)$. Fix any n.  Then
  $g_{|\Omega_n}\ge{f_\alpha}_{|\Omega_n}\uparrow f_{|\Omega_n}$ yields
  $g_{|\Omega_n}\ge f_{|\Omega_n}$. This is true for all $n$;
  therefore, $g\ge f$. It follows that $f=\sup_\alpha f_\alpha$.
\end{proof}

\begin{theorem}
  $C(\mathbb R^m)$ has the countable sup property.
\end{theorem}

\begin{proof}
  Let $(f_\alpha)_{\alpha\in\Lambda}$ be a net in $C(\mathbb R^m)$
  indexed by some directed set $\Lambda$ such that
  $f_\alpha\uparrow f$ for some $f\in C(\mathbb R^m)$.  Fix
  $n\in\mathbb N$. By Lemma~\ref{restr}, we have
  ${f_\alpha}_{|\Omega_n}\uparrow f_{|\Omega_n}$ in $C(\Omega_n)$.
  Recall that every vector lattice which admits a strictly positive
  functional has the countable sup property; see, e.g.,
  \cite[Theorem~2.6]{Aliprantis:03}. Since the Riemann integral is a
  strictly positive functional on $\Omega_n$, it follows that
  $C(\Omega_n)$ has the countable sup property. Therefore, the net
  $({f_\alpha}_{|\Omega_n})$ has a subsequence whose supremum in
  $C(\Omega_n)$ is $f_{|\Omega_n}$.

  We will now inductively construct a double sequence $(\alpha_{n,k})$ in
  $\Lambda$ indexed by pairs $(n,k)\in\mathbb N^2$ satisfying the
  following two conditions
  \begin{enumerate}
  \item $(\alpha_{n,k})$ is increasing: if $n_1\le n_2$ and $k_1\le
    k_2$ then $\alpha_{n_1,k_1}\le\alpha_{n_2,k_2}$;
  \item for every $n$ we have ${f_{\alpha_{n,k}}}_{|\Omega_n}\uparrow
    f_{|\Omega_n}$ as $k\to\infty$.
  \end{enumerate}
  For $n=1$, let $(\alpha_{1,k})_{k=1}^\infty$ be an increasing
  sequence in $\Lambda$ such that ${f_{\alpha_{1,k}}}_{|\Omega_1}\uparrow
    f_{|\Omega_1}$ as $k\to\infty$; it exists by the preceding
    paragraph.

  Suppose that we have constructed $\alpha_{i,k}$ for all
  $i=1,\dots,n-1$ and all $k$. We will now construct $\alpha_{n,k}$
  for all $k$. Find an increasing sequence
  $(\beta_k)$ in $\Lambda$ such that
  ${f_{\beta_k}}_{|\Omega_n}\uparrow f_{|\Omega_n}$ in
  $C(\Omega_n)$ as $k\to\infty$. Now choose $\alpha_{n,1}$ to be any index greater
  than $\beta_1$ and $\alpha_{n-1,1}$. Choose $\alpha_{n,2}$ to be any index greater
  than $\beta_2$, $\alpha_{n-1,2}$, and $\alpha_{n,1}$. Proceed
  inductively: once we have defined
  $\alpha_{n,1},\dots,\alpha_{n,k-1}$, choose $\alpha_{n,k}$ to be any
  index greater than $\beta_k$, $\alpha_{n-1,k}$, and
  $\alpha_{n,k-1}$. This yields
  \begin{displaymath}
    f_{|\Omega_n}\ge {f_{\alpha_{n,k}}}_{|\Omega_n}\ge {f_{\beta_k}}_{|\Omega_n}\uparrow f_{|\Omega_n},
  \end{displaymath}
  so that ${f_{\alpha_{n,k}}}_{|\Omega_n}\uparrow f_{|\Omega_n}$ in $C(\Omega_n)$ as
  $k\to\infty$.  Also, it is easy to see that $(\alpha_{n,k})$ is
  increasing.

  Consider $(\alpha_{k,k})_{k=1}^\infty$. This is clearly an
  increasing sequence in $\Lambda$. If $n\le k$ then
  $\alpha_{n,k}\le\alpha_{k,k}$, so that
  $f_{\alpha_{n,k}}\le f_{\alpha_{k,k}}\le f$. This yields
  ${f_{\alpha_{k,k}}}_{|\Omega_n}\uparrow f_{|\Omega_n}$ in
  $C(\Omega_n)$ as $k\to\infty$. It now follows from Lemma~\ref{restr}
  that $f_{\alpha_{k,k}}\uparrow f$ in $C(\mathbb R^m)$.
\end{proof}

\begin{corollary}\label{odense-pos}
  For every $0\le f\in C(\mathbb R^m)$ there exists a sequence $(h_n)$
  of positive piecewise affine functions such that $h_n\uparrow f$ in
  $C(\mathbb R^m)$.
\end{corollary}

Thus, Question~\ref{q:seq} has the affirmative answer. We now pass to
Question~\ref{q:non-pos}. We answer it through order convergence. In
literature, there are two slightly different definitions of an order
convergent net in a vector lattice, which are, generally, not
equivalent. We refer the reader to
\cite{Anderson:67,Abramovich:02,Abramovich:05}. In our case, it will
be irrelevant which of the two definitions is used for the following
two reasons: first, both definitions yield the properties that we will
be using and, second, by~\cite{Anderson:67}, the two definitions agree
in vector lattices which have the countable sup property; hence, they
agree in $C(\mathbb R^m)$.

\begin{proposition}
  Every function in $C(\mathbb R^m)$ is the order limit of an order
  convergent sequence of piecewise affine functions.
\end{proposition}

\begin{proof}
  Let $f\in C(\mathbb R^m)$. Decompose it as $f=f^+-f^-$. By
  Corollary~\ref{odense-pos}, there exist positive sequences $(g_n)$
  and $(h_h)$ in $\mathcal{PA}$ such that $g_n\uparrow f^+$ and
  $h_n\uparrow f^-$ in $C(\mathbb R^m)$. It follows that the sequence
  $(g_n-h_n)$ converges in order to $f$ in $C(\mathbb R^m)$.
\end{proof}

\bigskip
\textbf{Acknowledgement.} We would like to thank Foivos Xanthos for
valuable discussions.

\end{document}